\documentclass[USenglish,envcountsame]{llncs}

\usepackage[unicode=true,colorlinks=true]{hyperref}

\usepackage{amsmath,amssymb}
\let\doendproof\endproof
\renewcommand\endproof{~\hfill\qed\doendproof}

\usepackage[latin1]{inputenc}
\usepackage{enumerate,enumitem}
\usepackage{cite}

\usepackage[capitalize,nameinlink]{cleveref}

\usepackage{graphicx}
\DeclareGraphicsExtensions{.pdf,.png,.eps}
\graphicspath{{figs/}}

\newcommand{\pn}{\ensuremath{\operatorname{pn}}}
\newcommand{\sa}{\ensuremath{\operatorname{sa}}}
\newcommand{\cn}{\ensuremath{\operatorname{cn}}}
\newcommand{\arb}{\ensuremath{\operatorname{a}}}
\newcommand{\mad}{\ensuremath{\operatorname{mad}}}
\newcommand{\calP}{\ensuremath{\mathcal{P}}}

\let\leq\leqslant
\let\geq\geqslant
\let\subset\subseteq

\begin{document}

\title{Local and Union Page Numbers}

\author{
 {Laura Merker} \and 
 {Torsten Ueckerdt}
}


\institute{
 Karlsruhe Institute of Technology (KIT), 
 Institute of Theoretical Informatics, Germany\\
 \email{laura.merker@student.kit.edu}, \email{torsten.ueckerdt@kit.edu}
}

\maketitle

\begin{abstract}
 We introduce the novel concepts of local and union book embeddings, and, as the corresponding graph parameters, the local page number $\pn_\ell(G)$ and the union page number $\pn_u(G)$.
 Both parameters are relaxations of the classical page number $\pn(G)$, and for every graph $G$ we have $\pn_\ell(G) \leq \pn_u(G) \leq \pn(G)$.
 While for $\pn(G)$ one minimizes the total number of pages in a book embedding of $G$, for $\pn_\ell(G)$ we instead minimize the number of pages incident to any one vertex, and for $\pn_u(G)$ we instead minimize the size of a partition of $G$ with each part being a vertex-disjoint union of crossing-free subgraphs. 
 While $\pn_\ell(G)$ and $\pn_u(G)$ are always within a multiplicative factor of $4$, there is no bound on the classical page number $\pn(G)$ in terms of $\pn_\ell(G)$ or $\pn_u(G)$.
 
 We show that local and union page numbers are closer related to the graph's density, while for the classical page number the graph's global structure can play a much more decisive role.
 We introduce tools to investigate local and union book embeddings in exemplary considerations of the class of all planar graphs and the class of graphs of tree-width $k$.
 As an incentive to pursue research in this new direction, we offer a list of intriguing open problems.
 
\end{abstract}

\section{Introduction}
\label{sec:introduction}

A \emph{linear embedding} of a graph $G = (V,E)$ is a tuple $(\prec,\calP)$ where $\prec$ is a total ordering\footnote{We define $\prec$ as a linear ordering. However, in a few places we shall think of $\prec$ as a cyclic ordering. This is legitimate as we are interested in crossing edges only, and these are preserved under cyclic shifts.} of the vertex set $V$ and $\calP = \{P_1,\ldots,P_k\}$ is a partition of the edge set $E$.
The ordering $\prec$ is sometimes called the \emph{spine ordering}, and each part $P_i$ of $\calP$ is called a \emph{page}.
For a given spine ordering $\prec$, two edges $uv, xy \in E$ with $u \prec v$ and $u \prec x \prec y$ are said to be \emph{crossing} if $u \prec x \prec v \prec y$.
A linear embedding $(\prec,\calP)$ is a \emph{book embedding} if for any two edges $uv$ and $xy$ in $E$ we have
\begin{equation}
 \text{if } u \prec x \prec v \prec y \text{ and } uv \in P_i, xy \in P_j \text{ then } i \neq j.\label{eq:non-crossing}
\end{equation}
So \cref{eq:non-crossing} simply states that no two edges in the same page are crossing, or equivalently, any two crossing edges are assigned to distinct pages in $\calP$.

Book embeddings were introduced by Ollmann~\cite{Oll-73} as well as Bernhart and Kainen~\cite{BK-79}, see also~\cite{Kai-74}.
Besides their apparent applications in real-world problems (see e.g.~\cite{Ros-83,CLR-87} and the numerous references in~\cite{DW-04}), book embeddings enjoy steady popularity in graph theory; see for example~\cite{VWY-09,ENO-97,HI-92,TY-02,Mal-94,Yan-89}, just to name a few.
In most cases (also including the generalizations for directed graphs~\cite{BDDDMP-19} or pages with limited crossings~\cite{BDHL-18}), one seeks to find a book embedding with as few pages as possible for given graph $G$.
In particular, $(\prec,\calP)$ is a \emph{$k$-page book embedding} if $|\calP| = k$, and the \emph{page number} of $G$, denoted by $\pn(G)$, is the smallest $k$ for which we can find a $k$-page book embedding of $G$.
(We remark that $\pn(G)$ is sometimes also called the book thickness~\cite{BK-79} or stack number~\cite{DW-04} of $G$.)

As the main contribution of the present paper, we propose two relaxations of the page number parameter: The local page number $\pn_\ell(G)$ and the union page number $\pn_u(G)$.
We initiate the study of these parameters by comparing $\pn_\ell(G)$, $\pn_u(G)$, and $\pn(G)$ for graphs $G$ in some natural graph classes, such as planar graphs (c.f.~\cref{sec:planars}), graphs of bounded density (c.f.~\cref{sec:density}), and graphs of bounded tree-width (c.f.~\cref{sec:tree-width}).
Besides these bounds, a (perhaps not surprising) result showing computational hardness (c.f.~\cref{thm:hardness}), and a few structural observations, we also give some intriguing open problems at the end of the paper in \cref{sec:conclusions}.

Before listing our specific results in \cref{subsec:our-results} below, let us define and motivate the novel parameters local and union page numbers.

\medskip

\noindent
\textbf{Local Page Numbers.}\;
For a book embedding $(\prec,\calP)$ of graph $G = (V,E)$ and a vertex $v \in V$, let us denote by $\calP_v$ the subset of pages that contain at least one edge incident to $v$.
Then we define:
\begin{itemize}
 \item A book embedding is \emph{$k$-local} if $|\calP_v| \leq k$ for each $v \in V$, i.e., each vertex has incident edges on at most $k$ pages.
 
 \item The \emph{local page number}, denoted by $\pn_\ell(G)$, is the smallest $k$ for which we can find a $k$-local book embedding of $G$.
\end{itemize}
Thus, we seek to find a book embedding with any number of pages (possibly more than $\pn(G)$), but with no vertex having incident edges on more than $k$ of these pages.
As each $k$-page book embedding is a $k$-local book embedding,
\begin{equation}
 \text{for any graph } G \text{ we have } \qquad \pn_\ell(G) \leq \pn(G).
\end{equation}
However, $\pn_\ell(G)$ can be strictly smaller than $\pn(G)$.
For example, $K_5$ and $K_{3,3}$ both have page number $3$ and local page number $2$.
As illustrated in \cref{fig:K33-and-K5}, $K_5$ admits a $2$-local $3$-page book embedding, i.e., this book embedding simultaneously certifies $\pn(K_5) \leq 3$ and $\pn_\ell(K_5) \leq 2$.
In the left of \cref{fig:K33-and-K5} we have a $2$-local $4$-page book embedding of $K_{3,3}$ (when the three orange/thick edges are put into three separate pages).
So here, the introduction of ``extra'' pages, additionally to the necessary $\pn(K_{3,3}) = 3$ pages in every book embedding of $K_{3,3}$, allowed us to actually reduce the maximum number of pages incident to any one vertex from $3$ to $\pn_\ell(K_{3,3}) = 2$.
And for some graphs $G$ with $\pn_\ell(G) = k$, in fact all $k$-local book embeddings have more than $\pn(G)$ pages.

\begin{figure}[t]
 \centering
 \includegraphics{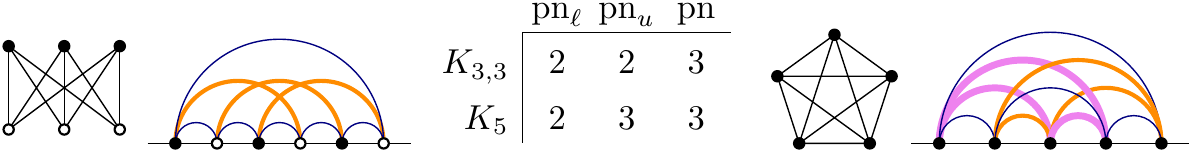}
 \caption{Comparison of local, union, and classical page numbers on the examples of $K_{3,3}$ and $K_5$.
 Left: $2$-page union embedding of $K_{3,3}$.
 Right: $2$-local book embedding of $K_5$.}
 \label{fig:K33-and-K5}
\end{figure}

\medskip

\noindent
\textbf{Union Page Numbers.}\;
For a linear embedding $(\prec,\calP)$ (so not necessarily a book embedding) of graph $G = (V,E)$ and a page $P \in \calP$, let us denote by $G_P$ the subgraph of $G$ on all edges in $P$ and all vertices with some incident edge in $P$.
Then we define:
\begin{itemize}
 \item A linear embedding $(\prec,\calP)$ is a \emph{union embedding} if $(\prec,\{E(C)\})$ is a ($1$-page) book embedding for each connected component $C$ of $G_P$ and each $P \in \calP$, i.e., each connected component of each page is crossing-free.
 
 \item The \emph{union page number}, denoted by $\pn_u(G)$, is the smallest $k$ for which we can find a $k$-page union embedding of $G$.
\end{itemize}
In other words, in a union embedding, each page is the vertex-disjoint union of crossing-free graphs; hence the name ``union page number''.
So we allow crossing edges on a single page $P$, as long as these are contained in different components of $G_P$.
For the union page number $\pn_u(G)$ we minimize the number of pages, just like for the classical page number $\pn(G)$.

Again, each $k$-page book embedding is also a $k$-page union embedding, giving $\pn_u(G) \leq \pn(G)$.
Moreover, each $k$-page union embedding can be transformed into a $k$-local book embedding by putting each component of each page onto a separate page, giving $\pn_\ell(G) \leq \pn_u(G)$.
Summarizing,
\begin{equation}
 \text{for any graph } G \text{ we have } \qquad \pn_\ell(G) \leq \pn_u(G) \leq \pn(G).\label{eq:inequalities}
\end{equation}
Consider again the linear embedding of $K_{3,3}$ in the left of \cref{fig:K33-and-K5}, but this time put all three orange/thick edges on the same page $P$.
These edges are pairwise crossing, so this is not a book embedding. 
However these edges lie in separate connected components of $G_P$, so this is a union embedding.
As we found a $2$-page union embedding of $K_{3,3}$, we see $\pn_u(K_{3,3}) \leq 2 < 3 = \pn(K_{3,3})$.

Comparing union and local page numbers, we have that $\pn_\ell(G)$ can be strictly smaller than $\pn_u(G)$.
For example, we have already seen in \cref{fig:K33-and-K5} that $\pn_\ell(K_5) \leq 2$, and we claim that $\pn_u(K_5) > 2$.
Indeed, for the cyclic spine ordering $v_1 \prec \cdots \prec v_5$ and pages $P_1,P_2$ we may assume by symmetry that  $v_1v_3, v_1v_4, v_2v_5 \in P_1$ and $v_2v_4, v_3v_5 \in P_2$.
As each connected component of $G_{P_1}$ and $G_{P_2}$ is crossing-free, $v_2$ and $v_3$ are in distinct components in both page $P_1$ and page $P_2$, leaving no way to assign the edge $v_2v_3$.

\medskip
\noindent
\textbf{Motivation.}\;
Local and union page numbers are motivated by local and union covering numbers as introduced by Knauer and the second author~\cite{KU-16}.
In order to give a brief summary of the covering number framework, consider a graph $H$ and a graph class $\mathcal{G}$.
An \emph{injective $\mathcal{G}$-cover of $H$} is a set $S = \{G_1,\ldots,G_m\}$ of subgraphs\footnote{In a general $\mathcal{G}$-cover one considers graph homomorphisms from graphs in $\mathcal{G}$ into $H$. However, we consider here only injective $\mathcal{G}$-covers, which is equivalent to considering subgraphs of $H$.} of $H$ such that $H = G_1 \cup \cdots \cup G_m$ and $G_i \in \mathcal{G}$ for $i=1,\ldots,m$.
In other words, $H$ is covered by (is the union of) some $m$ (possibly isomorphic) graphs from $\mathcal{G}$.
Moreover, let $\overline{\mathcal{G}}$ denote the class of all finite vertex-disjoint unions of graphs in $\mathcal{G}$, meaning that $G \in \overline{\mathcal{G}}$ if and only if $G$ is the vertex-disjoint union of some number of graphs in $\mathcal{G}$.

The \emph{global $\mathcal{G}$-covering number} of $H$, denoted by $\cn_g^{\mathcal{G}}(H)$, is the smallest $m$ such that there exists an injective $\mathcal{G}$-cover of $H$ of size $m$, i.e., using $m$ graphs in $\mathcal{G}$.
The \emph{union $\mathcal{G}$-covering number} of $H$, denoted by $\cn_u^{\mathcal{G}}(H)$, is the smallest $m$ such that there exists an injective $\overline{\mathcal{G}}$-cover of $H$ of size $m$, i.e., using $m$ vertex-disjoint unions of graphs in $\mathcal{G}$.
The \emph{local $\mathcal{G}$-covering number} of $H$, denoted by $\cn_\ell^\mathcal{G}(H)$, is the smallest $k$ such that there exists an injective $\mathcal{G}$-cover of $H$ in which every vertex of $H$ is contained in at most $k$ graphs of the cover, i.e., using any number of graphs from $\mathcal{G}$ but with no vertex of $H$ being contained in more than $k$ of these.\footnote{The covering number framework includes a fourth covering number, the folded $\mathcal{G}$-covering number of $H$, which we omit here, so as not to congest the discussion.}

Many graph parameters (including arboricities, thickness parameters, variants of chromatic numbers, several Ramsey numbers, and interval representations) are $\mathcal{G}$-covering numbers of a certain type and for a certain graph class $\mathcal{G}$.
Moreover, recently the global-union-local framework was extended to settings that do not directly concern graph covers, such as the local and union boxicity~\cite{BSU-18}, and the local dimension of posets~\cite{Uec-16}, which has stimulated research drastically~\cite{MM-18,FU-19,TW-17,KMMSSUW-18,BGT-17,BPSTT-17}.
Our proposed local and union page numbers naturally arise from the covering number framework by using ordered graphs and ordered subgraphs in the above definitions and taking $\mathcal{G}$ to be the class of all crossing-free ordered graphs.

Particularly the local page number might be very useful in applications.
For example, oftentimes the spine ordering $\prec$ of $G$ is already given from the problem formulation (by time stamps, geographic positions or a genetic sequence).
Then the edges of $G$ model some kind of connections and classical book embeddings are used to distribute the connections to machines that can process sets of connections that satisfy the LIFO (last-in-first-out) property.
Local book embeddings could be used to model situations in which the total number of machines is not the scarce resource but rather the number of machines working on the same element, i.e., vertex.
Imagine for example limited capacity at each element in terms of computing power (as for cell phones) or simply spatial restrictions (as for genes).
This kind of task is precisely modeled by local book embeddings and the local page numbers.


\subsection{Our Contribution}
\label{subsec:our-results}

First, we show that the new parameters $\pn_\ell(G)$ and $\pn_u(G)$ can be arbitrarily smaller than the classical page number $\pn(G)$, while local and union page number are always at most a multiplicative factor of $4$ apart.

\begin{theorem}\label{thm:separation}
 For any $k \geq 3$ and infinitely many values of $n$, there exist $n$-vertex graphs $G$ with
 \[
  \pn_\ell(G) \leq \pn_u(G) \leq k+2 \quad \text{and} \quad \pn(G) = \Omega\left(\sqrt{k}n^{1/2-1/k}\right).
 \]
 In contrast, for every graph $G$ we have $\pn_u(G) \leq 4 \pn_\ell(G)+2$.
\end{theorem}

While for every planar graph $G$ we have $\pn(G) \leq 4$~\cite{Yan-89}, it is not known whether there is a planar graph $G$ with $\pn(G) = 4$.
The best known lower bound was given by Bernhart and Kainen~\cite{BK-79}, who presented a planar graph $G$ with $\pn(G) = 3$.
That very graph satisfies $\pn_\ell(G) = 2$, but we can augment it to a planar graph with local page number $3$.

\begin{theorem}\label{thm:planar-LB}
 There is a planar graph $G$ with $\pn_\ell(G) = 3$.
\end{theorem}

For graphs $G$ with tree-width~$k$, it is known that $\pn(G) \leq k$ if $k \in \{1,2\}$~\cite{RV-95} and $\pn(G) \leq k+1$ if $k \geq 3$~\cite{GH-01}, and both bounds are best possible~\cite{DW-07,VWY-09}.
For the local and union page number we get a lower bound of $k$.

\begin{theorem}\label{thm:tree-width-LB}
 For every $k \geq 1$ there is a graph $G$ of tree-width~$k$ with $\pn_u(G) \geq \pn_\ell(G) \geq k$.
\end{theorem}

Finally, it is known that $\pn(G) \leq 2$ if and only if $G$ is a subgraph of a planar Hamiltonian graph~\cite{BK-79}.
Hence, it follows from~\cite{Wig-82} that deciding $\pn(G) \leq 2$ is NP-complete, which easily generalizes to $\pn(G) \leq k$ for each $k \geq 2$.
(Since $\pn(G) = 1$ is equivalent to $G$ being outerplanar, this can be efficiently tested.)
If the spine ordering $\prec$ is already given, the problem of finding an edge partition into $k$ crossing-free pages is equivalent to that of properly $k$-coloring circle graphs and hence determining the smallest such $k$ is NP-complete~\cite{GJMP-80}.
While properly $k$-coloring circle graphs is polynomial-time solvable for $k=2$, it is open whether the problem becomes NP-hard for fixed $k \geq 3$.
For the local page number we have NP-completeness for fixed spine ordering $\prec$ and each fixed $k \geq 3$.

\begin{theorem}\label{thm:hardness}
 For any $k \geq 3$, it is NP-complete to decide for a given graph $G$ and given spine ordering $\prec$, whether there exists an edge partition $\calP$ such that $(\prec,\calP)$ is a $k$-local book embedding.
\end{theorem}

For a proof of \cref{thm:hardness} we refer the interested reader to the Bachelor's thesis of the first author~\cite{Mer-18}.

\section{Bounds in Terms of Density}
\label{sec:density}

Though not a fixed mathematical concept, the density of a graph $G = (V,E)$ quantifies the number $|E|$ of edges in terms of the number $|V|$ of vertices.
An important specification of density is the \emph{maximum average degree} of $G$ defined by
\[
 \mad(G) = \max\left\{ \frac{2|E(H)|}{|V(H)|} \mid H \subseteq G, H \neq \emptyset\right\}.
\]
Recall that for a linear embedding $(\prec,\calP)$ of $G = (V,E)$ and a page $P \in \calP$ we denote by $G_P = (V_P,P)$ the subgraph of $G$ on all edges in $P$ and all vertices of $G$ with at least one incident edge in $P$.
Clearly, if $P$ is crossing-free, then $G_P$ is outerplanar and thus $|P| \leq 2|V_P|-3$.
As $\bigcup_{P \in \calP} P = E$ and $V_P \subseteq V$ for each page $P$, we immediately get an upper bound on the density of any graph with a $k$-local book embedding.

\begin{lemma}\label{lem:density-bound}
 For any graph $G = (V,E)$ we have
 \[
   \pn_\ell(G) \geq \max\left\{ \frac{|E(H)|}{2|V(H)|-3} \mid H \subseteq G, H \neq \emptyset\right\}.
 \]
\end{lemma}
\begin{proof}
 Let $H$ be any non-empty subgraph of a graph $G$ of local page number $\pn_\ell(G) = k$.
 Then there is a $k$-local book embedding $(\prec,\calP)$ of $H$, each page $P$ of which describes an outerplanar graph $H_P = (V_P,P)$.
 Thus
 \[
  |E(H)| \leq \sum_{P \in \calP} \left(2|V_P|-3 \right) \leq 2k|V(H)| - 3|\calP| \leq \pn_\ell(G) \cdot (2|V(H)|-3).
 \]
\end{proof}

\noindent
From \cref{lem:density-bound} and \cref{eq:inequalities} we conclude for every graph $G$ that
\begin{equation}
 \pn(G) \geq \pn_u(G) \geq \pn_\ell(G) \geq \mad(G)/4.
\end{equation}
In other words, the graph's density gives a lower bound on all three kinds of page numbers.
Perhaps surprisingly, there is also an \emph{upper} bound on the union and local page numbers in terms of the graph's density.

Nash-Williams~\cite{Nas-64} proved that any graph $G$ edge-partitions into $k$ forests if and only if
\[
 k \geq \max\left\{ \frac{|E(H)|}{|V(H)|-1} \mid H \subseteq G, |V(H)| \geq 2 \right\}.
\]
The smallest such $k$, the \emph{arboricity} $\arb(G)$ of $G$, thus satisfies $\frac{1}{2}\mad(G) < \arb(G) \leq \frac{1}{2}\mad(G)+1$.
The \emph{star arboricity} $\sa(G)$ of $G$ is the minimum $k$ such that $G$ edge-partitions into $k$ star forests.
Thus $\sa(G)$ is the union $\mathcal{G}$-covering number of $G$ with respect to the class $\mathcal{G} = \{K_{1,n} \mid n \in \mathbb{N}\}$ of all stars.
Using the covering number framework, Knauer and the second author~\cite{KU-16} introduced the corresponding local $\mathcal{G}$-covering number, the \emph{local star arboricity} $\sa_\ell(G)$, as the minimum $k$ such that $G$ edge-partitions into some number of stars, but with each vertex having an incident edge in at most $k$ of these stars.
It is known~\cite{AMR-92,KU-16} that $\sa(G)$ and $\sa_\ell(G)$ can be bound in terms of $\arb(G)$ as
\[
 \arb(G) \leq \sa(G) \leq 2\arb(G) \qquad \text{ and } \qquad \arb(G) \leq \sa_\ell(G) \leq \arb(G)+1.
\]

\begin{theorem}\label{thm:bounds-from-stars}
 For any graph $G$ we have $\pn_\ell(G) \leq \sa_\ell(G)$ and $\pn_u(G) \leq \sa(G)$.
 In particular, we have
 \[
  \frac{\mad(G)}{4} \leq \pn_\ell(G) \leq \frac{\mad(G)}{2}+2 \quad \text{ and } \quad \frac{\mad(G)}{4} \leq \pn_u(G) \leq \mad(G)+2.
 \]
\end{theorem}
\begin{proof}
 Take an arbitrary spine ordering $\prec$ and an edge-partition $\calP$ into stars.
 Then each page is crossing-free, which shows $\pn_\ell(G) \leq \sa_\ell(G)$.
 Now take an arbitrary spine ordering $\prec$ and an edge-partition $\calP$ into star forests.
 Then each connected component on each page is a star and thus crossing-free, which shows $\pn_u(G) \leq \sa(G)$.
\end{proof}

Though \cref{thm:bounds-from-stars} is merely an observation, it has a number of interesting consequences.
First of all, the local and union page number are not too far apart: $\pn_\ell(G) \leq \pn_u(G) \leq 4\pn_\ell(G)+2$.
However, the local and union page numbers can be very far from the classical page number.
For example, we have $\sa_\ell(G) \leq k = \mad(G)$ for every $k$-regular graph $G$, and hence $\pn_\ell(G) \leq k$ and $\pn_u(G) \leq k+2$ whenever $G$ is $k$-regular.
On the other hand, Malitz~\cite{Mal-94} proved that for every $k \geq 3$ there are $n$-vertex $k$-regular graphs $G$ with page number $\pn(G) = \Omega(\sqrt{k}n^{1/2-1/k})$.
Together this proves \cref{thm:separation}.

For planar $G$ we have $\arb(G) \leq 3$~\cite{Nas-64}, hence $\sa_\ell(G) \leq 4$~\cite{KU-16}, as well as $\sa(G) \leq 5$~\cite{HMS-96}.
Hence, \cref{thm:bounds-from-stars} immediately gives the following (without relying on Yannakakis' result~\cite{Yan-89}).

\begin{corollary}\label{cor:planar}
 For every planar graph $G$ we have $\pn_\ell(G) \leq 4$ and $\pn_u(G) \leq 5$.
\end{corollary}

\section{Planar Graphs}
\label{sec:planars}

In this section we consider planar graphs.
In particular, we prove \cref{thm:planar-LB} stating the existence of a planar graph with local page number $3$.
Our planar graph will be a large enough stacked triangulation (also known as planar $3$-trees, chordal triangulations, or Apollonian networks).
For this let $T_0 \cong K_3$ and for $n \geq 1$ define $T_n$ as obtained from $T_{n-1}$ by placing a new vertex $v_\Delta$ in each facial triangle $\Delta$ of $T_{n-1}$, and connecting $v_\Delta$ by edges to each of the three vertices of $\Delta$.
Thus, for $n \geq 0$ we have $|V(T_n)| = 3^n + 2$.

Suppose for the sake of contradiction there is a $2$-local book embedding $(\prec,\calP)$ of $T_9$.
We consider the subgraphs $T_0 \subset T_1 \subset \cdots \subset T_9$ of $T_9$.

\begin{claim}
 There exists an edge $vw$ in $T_1$ with $\calP_v = \calP_w$ and $|\calP_v| = |\calP_w| = 2$.
\end{claim}
Indeed, consider the four vertices $v_1,v_2,v_3,v_4$ of one of the two $K_4$ subgraphs in $T_1$.
Without loss of generality assume that $|\calP_{v_1}| = \cdots = |\calP_{v_4}| = 2$.
As $\calP_{v_i} \cap \calP_{v_j} \neq \emptyset$ for any $i,j \in \{1,\ldots,4\}$, we can see $\calP_{v_1},\ldots,\calP_{v_4}$ as four pairwise incident edges in a multigraph $I$ on vertex set $\calP$, where two vertices of $I$ are connected by an edge if there is common vertex of $G$ on the two respective pages.
Thus, if $\calP_{v_1},\ldots,\calP_{v_4}$ were pairwise distinct, they would form a star, i.e., all pairwise intersections would be the same page $P \in \calP$.
But then the whole $K_4$ subgraph on $v_1,\ldots,v_4$ would be embedded on page $P$, which is impossible as $K_4$ is not outerplanar.

\medskip

So let $vw$ be an edge in $T_1$ with $\calP_v = \calP_w = \{P_1,P_2\}$.
By the inductive construction of stacked triangulations, there is a set $X = \{x_1,\ldots,x_7\}$ of seven vertices in $T_8 - T_1$ that are incident to $v$ and $w$ and induce a path in $T_9$; see \cref{fig:planar-LB-illustration}.
By pigeon-hole principle and cyclic shifts of $\prec$, we may assume that $v \prec x_1 \prec x_2 \prec x_3 \prec x_4 \prec w$, where $x_1,\ldots,x_4$ are consecutive in $\prec$ when restricted to $X$.
Each of $vx_i$ and $wx_i$, $i=1,\ldots,4$, lies on $P_1$ or $P_2$; say $vx_4 \in P_1$.
Then $wx_1,wx_2,wx_3 \in P_2$, and thus $vx_2,vx_3 \in P_1$.
In particular, we have $\calP_{x_2} = \calP_{x_3} = \{P_1,P_2\}$.

\begin{figure}[t]
 \centering
 \includegraphics{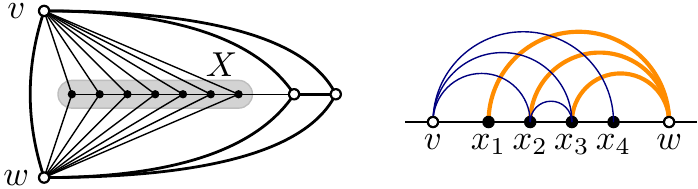}
 \caption{Part of the planar graph with local page number $3$ (left) and part of the hypothetical $2$-local book embedding (right).}
 \label{fig:planar-LB-illustration}
\end{figure}

Now observe that $x_2$ cannot be adjacent to any vertex $y$ with $x_3 \prec y$ and $y \neq w$.
Indeed, such an edge $x_2y$ would cross the edge $vx_3 \in P_1$ and one of $wx_1,wx_3 \in P_2$.
Symmetrically, $x_3$ cannot be adjacent to any vertex $y$ with $y \prec x_2$ and $y \neq v$.
As $X$ induces a path in $T_9$ and no vertex of $X$ lies between $x_2$ and $x_3$ in $\prec$, it follows that $x_2x_3$ is an edge of the path.
By symmetry assume $x_2x_3 \in P_1$.
This implies that $v$ cannot be adjacent to any vertex $y$ with $x_2 \prec y \prec x_3$, as such an edge $vy$ would cross the edges $x_2x_3 \in P_1$ and $wx_2 \in P_2$.

But then $v,x_2,x_3$ form a facial triangle $\Delta$ of $T_8$ with all three edges on page $P_1$.
However, there is no possible placement for the vertex $y = v_\Delta$ in $T_9$ that is adjacent to each of $v,x_2,x_3$.
Thus, the planar graph $T_9$ admits no $2$-local book embedding, which proves \cref{thm:planar-LB}.

\section{Graphs with Bounded Tree-Width}
\label{sec:tree-width}

In this section we investigate the largest union page number and the largest local page number among all graphs of tree-width $k$.
Clearly it suffices to consider edge-maximal graphs of tree-width $k$, the so-called $k$-trees, which are inductively defined as follows:
A graph $G$ is a \emph{$k$-tree} if and only if $G \cong K_{k+1}$ or $G$ is obtained from a smaller $k$-tree $G'$ by adding one new vertex $v$ whose neighborhood in $G'$ is a clique of order $k$.

As our main tool in this section, let us define a linear embedding $(\prec,\calP)$ to be a \emph{forest embedding} if the edges on each page $P \in \calP$ form a forest.
For a graph $G$, we say that a book embedding $(\prec,\calP)$ of some other graph $\bar{G} = (\bar{V},\bar{E})$ \emph{contains} a forest embedding of $G$ if there exists a set $X \subseteq \bar{V}$ such that $G \cong \bar{G}[X]$ and $(\prec,\calP)$ restricted to $\bar{G}[X]$ is a forest embedding of $G$.

\begin{lemma}\label{lem:forest-embedding}
 For every $\ell \in \mathbb{N}$ and every $k$-tree $G$ there exists a $k$-tree $\bar{G}$ such that every $\ell$-local book embedding of $\bar{G}$ contains a forest embedding of $G$.
\end{lemma}
\begin{proof}
 We find $\bar{G}$ based on $G = (V,E)$ by induction on $|V|$ as follows.

 In the base case we have $G \cong K_{k+1}$ and we find $\bar{G}$ by induction on $k$.
 In the base case of this inner induction we have $k=1$ and it suffices (for any $\ell$) to take $\bar{G} = G \cong K_2$.
 For $k > 1$, we get from induction a $(k-1)$-tree $\bar{G}_{k-1}$ all of whose $\ell$-local book embeddings contain a forest embedding of $K_k$.
 Starting with $\bar{G}_{k-1}$, add for each $k$-clique $C$ in $\bar{G}_{k-1}$ an independent set $I_C$ of $3k^2\ell$ vertices, together with all possible edges between $C$ and $I_C$.
 The resulting graph has tree-width $k$ and hence can be augmented to a $k$-tree $\bar{G}$.
 Consider any book embedding $(\prec,\calP)$ of $\bar{G}$.
 The inherited book embedding of $\bar{G}_{k-1} \subset \bar{G}$ contains a forest embedding of $K_k$, i.e., we have a forest embedding of some $k$-clique $C$ in $\bar{G}_{k-1}$.
 If one vertex $v$ in $I_C$ has its $k$ incident edges on $k$ pairwise different pages, then we have a forest embedding of $C \cup v \cong K_{k+1}$, as desired.
 Otherwise, each vertex $v$ in $I_C$ has two incident edges on the same page in $\calP$ joining $v$ with two vertices in $C$.
 By pigeon-hole principle, for a set $I'$ of at least $|I_C| / k^2 = 3\ell$ vertices of $I_C$ these are the same two vertices $c,c'$ of $C$.
 Since each of $c,c'$ has incident edges on at most $\ell$ pages, again by pigeon-hole principle, one page in $\calP$ contains the edges between $c,c'$ and at least $|I'| / \ell = 3$ vertices in $I'$.
 However this is a contradiction as $K_{2,3}$ is not outerplanar.

 \smallskip
 
 Now for the induction step of the outer induction, assume that $G$ is a $k$-tree with $|V| > k+1$ vertices.
 Then $G$ is obtained from a $k$-tree $G'$ by adding one vertex $v$ whose neighborhood in $G'$ is a clique of order $k$.
 From induction we get a $k$-tree $\bar{G}'$ all of whose $\ell$-local book embeddings contain a forest embedding of $G'$.
 Now we can do the same argument as before:
 Obtain $\bar{G}$ from $\bar{G}'$ by adding for each $k$-clique $C$ in $\bar{G}'$ an independent set $I_C$ of size $3k^2\ell$, together with all possible edges between $C$ and $I_C$.
 Then any $\ell$-local book embedding of $\bar{G}$ induces an $\ell$-local book embedding of $\bar{G}'$, which hence contains a forest embedding of $G'$.
 Let $C$ be the $k$-clique in $G'$ that forms the neighborhood of $v$ in $G$.
 The same argumentation as above then shows that at least one vertex in $I_C$ has its $k$ incident edges to $C$ on $k$ distinct pages, giving the desired forest embedding of $G$.
 (Essentially, the only difference to the base case is that adding the independent sets to $\bar{G}'$ gives a full $k$-tree, since $\bar{G}'$ is already a $k$-tree.)
\end{proof}

Having \cref{lem:forest-embedding}, \cref{thm:tree-width-LB} (the existence of a $k$-tree with local page number $k$) follows with two simple edge counts.
\begin{itemize}
 \item[] If $G = (V,E)$ admits a $\ell$-local forest embedding $(\prec,\calP)$, then
 \begin{equation}
  |E| \leq \sum_{P \in \calP} (|V_P|-1) \leq \ell|V| - |\calP| \leq \ell(|V|-1).\label{eq:forest-count}
 \end{equation}

 \item[] If $G = (V,E)$ is a $k$-tree, then
 \begin{equation}
  |E| = k|V| - \binom{k+1}{2}.\label{eq:k-tree-count}
 \end{equation}
\end{itemize}

To prove \cref{thm:tree-width-LB}, we shall find for each $k \geq 1$ a $k$-tree whose local page number is at least $k$.
For $k=1$ there is nothing to show.
For $k \geq 2$, let $G_0 = (V,E)$ be any $k$-tree with $|V| > \binom{k+1}{2}-(k-1)$ (Note that this is a vertex count!) and let $G = \bar{G}_0$ be the corresponding $k$-tree given by \cref{lem:forest-embedding} for $\ell = k-1$.
Assuming for the sake of contradiction that $\pn_\ell(G) \leq k-1$, we obtain a $(k-1)$-local forest embedding $(\prec,\calP)$ of $G_0$.
Then
\[
 |E| \overset{\eqref{eq:forest-count}}{\leq} (k-1)(|V|-1) = k|V| - (|V| + (k-1)) < k|V|-\binom{k+1}{2} \overset{\eqref{eq:k-tree-count}}{=} |E|,
\]
a contradiction.
Hence $\pn_\ell(G) \geq k$, as desired.

\medskip

To end this section, let us also discuss some further implications of \cref{lem:forest-embedding}.
We leave it open whether every $k$-tree has local page number at most $k$, i.e., whether the lower bound in \cref{thm:tree-width-LB} is tight.
By \cref{lem:forest-embedding} this is equivalent to every $k$-tree admitting a $k$-local forest book embedding.
By putting each tree in each forest on a separate page, we even get a $k$-local forest embedding $(\prec,\calP)$ with a tree on each page.
Moreover, by \cref{eq:forest-count} and \cref{eq:k-tree-count} we would have $|\calP| \leq \binom{k+1}{2}$, i.e., no more than $\binom{k+1}{2}$ trees in total, while at most $k$ at any one vertex.

And we get a similar statement for the maximum union page number of $k$-trees.
Suppose $(\prec,\calP)$ is an $\ell$-union embedding of some graph, and that on all pages in $\calP$ together we have $m$ connected components.
Putting each connected component on a separate (new) page, we obtain an $\ell$-local book embedding $(\prec,\tilde{\calP})$ with $|\tilde{\calP}| = m$ pages.
Now if $\pn_u(G) \leq k$ for all $k$-trees, then \cref{lem:forest-embedding} implies that every $k$-tree even admits a $k$-union forest embedding.
Moreover, by \cref{eq:forest-count} and \cref{eq:k-tree-count} we get a forest embedding with $m \leq \binom{k+1}{2}$ trees in total, while having at most $k$ at any one vertex.

Specifically, in order to prove $\pn_\ell(G) \leq k$ for every $k$-tree $G$, our task is to find a partition $\calP$ of the edges in $G$ into at most $\binom{k+1}{2}$ trees, such that every vertex is contained in no more than $k$ of these trees, as well as a spine ordering $\prec$ for which each of the trees is non-crossing.
The first part has a very natural solution:
Every $k$-tree $G$ has chromatic number $k+1$ and admits a unique\footnote{Up to relabeling of color classes.} proper $(k+1)$-vertex coloring $\phi$.
Moreover, there are exactly $\binom{k+1}{2}$ pairs of colors in $\phi$, any pair of color classes induces a tree in $G$, and each vertex of $G$ is contained in exactly $k$ of these trees.
Hence every $k$-tree $G$ edge-partitions into $\binom{k+1}{2}$ trees with each vertex being contained in $k$ of these trees.
Note that in this cover, every $(k+1)$-clique in $G$ has all $\binom{k+1}{2}$ edges in pairwise distinct trees.

We have however not been able to prove (or disprove) the existence of a spine ordering $\prec$ under which no pair of color classes induces a crossing.
If such exists, it would show $\pn_\ell(G) \leq k$ for all $k \geq 1$ and $\pn_u(G) \leq k$ for $k$ odd and $\pn_u(G) \leq k+1$ for $k$ even.
Note that for the union page number we also need to group the $\binom{k+1}{2}$ trees into as few forests of vertex-disjoint trees as possible.
Due to the nature of our coloring, this is equivalent to properly edge-coloring $K_{k+1}$; hence the distinction on the parity of $k$.

\section{Conclusions and Open Problems}
\label{sec:conclusions}

In this paper we presented two novel graph parameters: the local page number $\pn_\ell(G)$ and the union page number $\pn_u(G)$.
Both parameters are weakenings of the classical page number $\pn(G)$ and we have $\pn_\ell(G) \leq \pn_u(G) \leq \pn(G)$.
Hence, one might be able to strengthen existing lower bounds of the form $\pn(G) \geq X$ by showing $\pn_u(G) \geq X$ or even $\pn_\ell(G) \geq X$.
On the other hand, one might be able to support conjectured upper bounds of the form $\pn(G) \leq X$ by showing the weaker bounds $\pn_\ell(G) \leq X$ or even $\pn_u(G) \leq X$.

\begin{figure}[t]
 \centering
 \includegraphics{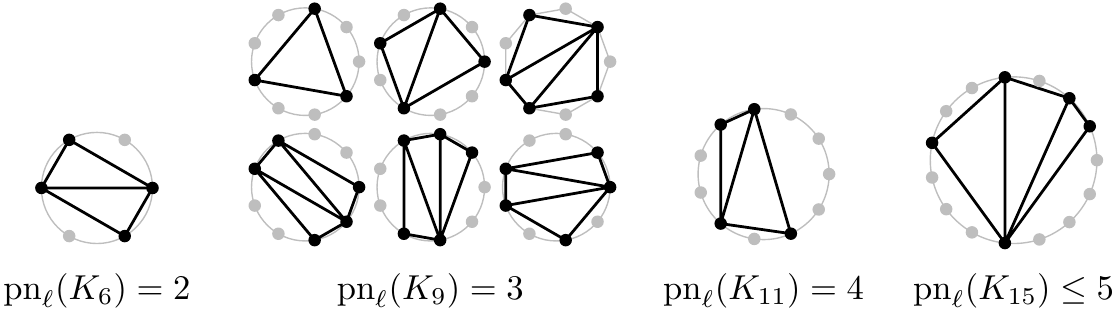}
 \caption{Illustrations of some $k$-local book embeddings of $K_n$ for $n=6,9,11,15$.
  The shown page for $n=6$ ($n=11$, $n=15$) is repeated $3$ times ($11$ times, $15$ times), each shifted cyclically by one position.}
 \label{fig:complete-graphs}
\end{figure}

In this paper we started to pursue this direction of research.
Let us list some concrete cases that are still open:

\begin{itemize}
 \item For the complete graph $K_n$ it is known~\cite{BK-79} that $\pn(K_n) = \lceil n/2\rceil$.
  On the other hand, the density of $K_n$ implies that $\pn_\ell(K_n) \geq \lceil (n-1)/4 \rceil$ (\cref{lem:density-bound}).
  In \cref{fig:complete-graphs} we indicate some $k$-local book embeddings of $K_n$ for some small values of $n$.
  According to this $\pn_\ell(K_6) \leq 2$, $\pn_\ell(K_9) \leq 3$, $\pn_\ell(K_{11}) \leq 4$, and $\pn_\ell(K_{15}) \leq 5$.
  Using the inequality $|E(G)| \leq 2\pn_\ell(G)|V(G)| - 3\pn(G)$ from the proof of \cref{lem:density-bound}, we see that $\pn_\ell(K_7) \geq 3$.
  (And with one further trick we get $\pn_\ell(K_{10}) \geq 4$.)
  We refer to~\cite{Mer-18} for more details, and state it is an open problem to improve the following general bounds:
  \[
   \left\lceil \frac{n-1}{4} \right\rceil \leq \pn_\ell(K_n) \leq \pn_u(K_n) \leq \pn(K_n) = \left\lceil \frac{n}{2} \right\rceil
  \]

 \item In 1989, Yannakakis~\cite{Yan-89} proved that for any planar graph $G$ we have $\pn(G) \leq 4$, while removing an earlier claim~\cite{Yan-86} that there would be some planar graph $G$ with $\pn(G) \geq 4$.
  Ganley and Heath~\cite{GH-01} observed that stacked triangulation $T_2$ (using our notation from \cref{sec:planars}, but also known as the Goldner-Harary graph) is a planar graph with $\pn(T_2) = 3$, which remains until today the best known lower bound.
  While $\pn_\ell(T_2) = 2$, we show in \cref{sec:planars} that $\pn_\ell(T_9) = 3$, while we leave it as an open problem to improve on the bounds
  \[
   3 \leq \max_{G \text{ planar}} \pn_\ell(G)  \leq \max_{G \text{ planar}} \pn_u(G) \leq \max_{G \text{ planar}} \pn(G) \leq 4.
  \]
  
  \item We have a similar open problem for $k$-trees, where we refer to the detailed discussion at the end of \cref{sec:tree-width}.
  \[
   k \leq \max_{G \text{ $k$-tree}} \pn_\ell(G) \leq \max_{G \text{ $k$-tree}} \pn_u(G) \leq \max_{G \text{ $k$-tree}} \pn(G) = \begin{cases}k & \text{if } k \leq 2\\ k+1 & \text{if } k \geq 3\end{cases}
  \]
\end{itemize}

Besides determining the local and union page numbers for other graph classes (like for example regular graphs), it is also interesting to further analyze the relation between $\pn_\ell(G), \pn_u(G), \arb(G)$ and $\sa(G)$.
For example, what is the maximum of $\pn_u(G) / \pn_\ell(G)$ over all graphs $G$?

Finally, let us mention that changing the non-crossing condition \cref{eq:non-crossing} underlying the notion of book embeddings to for example a non-nesting condition, we get local and union versions of queue numbers.
Interestingly, the proof of \cref{thm:bounds-from-stars} remains valid and so does \cref{cor:planar}, giving that every planar graph has local queue number at most~$4$ and union queue number at most~$5$.

\bibliographystyle{splncs04}
\bibliography{lit}

\end{document}